\documentclass[12pt]{article}
\usepackage{amsmath,amssymb,amsthm}

\usepackage{fullpage}

\title{ALGEBRAIC ADDITION THEOREMS}

\author{MARK B. VILLARINO \\[12pt]
        Escuela de matematica Universidad de Costa Rica}

\date{\today}

\newtheorem{theorem}{THEOREM}
\newtheorem{lemma}[theorem]{LEMMA}
\theoremstyle{definition}
\newtheorem{definition}{DEFINITION}
\newtheorem{example}{Example}

\newtheorem{coro}{COROLLARY}
 \newtheorem{rem}{Remark}
\newcommand{\shout}[1]{\emph{\textbf{#1}}} 
\numberwithin{equation}{subsubsection}
\numberwithin{equation}{subsection}

\renewcommand{\theenumi}{\roman{enumi}} 
\makeatletter
\renewcommand{\p@enumii}{\theenumi}   
\@addtoreset{equation}{section}
\makeatother
 
\newcommand{\C}{\mathbb{C}}           

\newcommand{\AAT}{{\small AAT}}       

\renewcommand{\:}{\colon}             

\begin{document}
\maketitle
\tableofcontents


\section{Introduction}

\indent

\textsc{Karl Weierstrass} was not only one of the greatest mathematicians of all time, but he was also one of the greatest teachers.  His lecture courses at the University of Berlin, beginning in 1855 and continuing every semester with almost no interruptions until his retirement in 1885 were justly famous as models of mathematical exposition.  

\begin{quote}  ``To perfection of form \textsc{Weierstrass} added that intangible something which is called inspiration . . . [and] he made creative mathematicians out of a disproportionately large fraction of his students."~\cite{Bell}\end{quote}

He customarily presented a four-semester cycle: complex function theory (1 semester), theory of elliptic functions (1 or 2 semesters), and the theory of algebraic integrals, functions, and the theory of the \textsc{Jacobi} inversion problem (1 or 2 semesters).  He would use different approaches at different times, reworking a subject over and over until finding the ``most natural" development of the theory.  This made for brilliant advanced mathematics courses and slow publication.

\textsc{Weierstrass}' theory of elliptic functions is a perfect example of this.  He proceeds as follows.  After discussing the exponential and trigonometric functions, say $\varphi(u)$, he concludes that they share the following properties:
\begin{enumerate}
  \item there exists an \emph{\textbf{algebraic equation}} between $\varphi(u)$, $\varphi(v)$, and $\varphi(u+v)$, whose coefficients do not depend on $u$ and $v$.
  \item the function $\varphi(u)$ is \emph{\textbf{meromorphic}} in all of finite space.
   \item there exists an \emph{\textbf{algebraic equation}} between the function $\varphi(u)$ and its \emph{\textbf{first derivative}}, $\varphi'(u)$, whose coefficients do not depend on $u$.
   \item the function $\varphi(u)$ is \emph{\textbf{periodic}}.
\end{enumerate}

 \textsc{Weierstrass} adds that functions such as $\sqrt[n]{\varphi(u)}$, where $n$ is an integer $>0$, have the properties (i), (iii), and iv), while in (ii), one must write \emph{\textbf{algegbroid}} instead of \emph{\textbf{meromorphic}}.

Then \textsc{Weierstrass} separates out property (i) as the \emph{characteristic} property of the functions $\varphi(u)$ and formalizes it in the following definition.

\definition If an analytic function $\varphi(u)$ has the characteristic property \begin{equation}
\label{AAT}
\boxed{G[\varphi(u),\varphi(v),\varphi(u+v)]=0}
\end{equation}where the $G(U,V,W)$ is a polynomial in $U$, $V$, and $W$ with coefficients that do not depend on $u$ and $v$, we say that $\varphi(u)$ \emph{admits an \textbf{ALGEBRAIC ADDITION THEOREM (AAT)}}.
\\

This leads \textsc{Weierstrass} to pose the following general problem as the starting point of the entire theory:

\begin{quote}

\emph{\textbf{It is required to find all analytic functions which admit an ALGEBRAIC ADDITION THEOREM.}}
 
 \end{quote}

 The \emph{answer} to \textsc{Weierstrass}' fundamental question is the content of the following two theorems:

\begin{theorem}\label{1} \begin{quote}.

 A MEROMORPHIC function admits an algebraic addition theorem if and only if it is either\begin{enumerate}
  \item a RATIONAL function; or
  \item a PERIODIC function.
  \end{enumerate}\end{quote}\end{theorem}

\begin{theorem}\label{2} \begin{quote}.

An analytic  function  admits an algebraic addition theorem (AAT) if and only if it is the ROOT of an ALGEBRAIC EQUATION the coefficients of which are all RATIONAL functions of the same MEROMORPHIC function which also admits an algebraic additon theorem.
\end{quote}
 \end{theorem}

  The first published proof of Theorem \ref{1} and Theorem \ref{2} is by \textsc{Phragmen}~\cite{Phragmen} in 1884.  Later proofs were published by \textsc{Forsyth}~\cite{Forsyth} in 1893, \textsc{K\"obe}~\cite{Koebe} in 1905, 1914, \textsc{Falk}~\cite{Falk}, in 1910, and \textsc{Hancock}~\cite{Hancock}, in 1910.  These proofs are written in several languages, (English, French, German) and are, to say the least, not easily available.  The books by \textsc{Forsyth} and \textsc{Hancock} have been republished, but they contain many flaws and mistakes.  
  
  Our paper presents a self-contained elementary exposition of \textsc{Weierstrass} theory with complete proofs and once again makes this beautiful theory available to the modern reader.  The content of our paper is\begin{itemize}
  
    \item a proof of Theorem~\ref{1} 
  \item a proof of Theorem~ \ref{2}
  \item Some comments on the extension of the theory to several variables as well as some unsolved problems.
\end{itemize} 

 
 \section{Meromorphic Functions Having an AAT}
 
 \indent 
 
 The investigation of \emph{meromorphic} functions having an AAT was carried to completion by \textsc{Weierstrass} in his lectures.  It is to be regretted that these lectures have never been published.
 
 The proof is a beautiful illustration of his theory of analytic function.  It shows how the rigid algebraic structure of the AAT imposes regularity and order on the unbridled chaos of the distribution of values of a meromorphic in the neighborhood of an essential singularity.
 
 It offers a marvelously satisfying "explanation" of \emph{\textbf{periodicity}}.
 
 
 \subsection{\textsc{Weierstrass}' ``tiny" \textsc{Picard} theorem.}
 
 \indent
 
 The theorems of \textsc{Picard} on entire and meromorphic functions are justly famous.  Experts classify them according to size: the ``small" theorem which deals with entire functions, and the ``big" theorem which deals with meromorphic functions.  Their proofs are not simple since they involve either the theory of the elliptic modular function or the theory of normal families.
 
 It does not seem to be well-known that long before \textsc{Picard} published his theorem (1880), \textsc{Weierstrass} had already generalized his own justly famous theorem on the behavior of an analytic function in the neighborhood of an isolated essential singularity to a partial version of the later \textsc{Picard} theorems.  Since these latter are called ``small" and ``big", we have called the \textsc{Weierstrass} result \textsc{Weierstrass}' ``tiny" \textsc{Picard} theorem.   We follow \textsc{Osgood}'s treatment~\cite{Osgood}.
 
 \begin{theorem}[\textsc{Weierstrass}' ``tiny" \textsc{Picard} theorem]\label{tiny} \begin{quote}.
 
 Suppose that the holomorphic function  $\varphi(u)$ has an isolated essential singularity at the point $u=a$.  Moreover, suppose \begin{enumerate}
  \item $C$ is an arbitrary complex number;
  \item $|w- C|<h$ is an arbitrarily small neighborhood of the point $w=C$.
\end{enumerate}\
Then: within this neighborhood there exists a point $C'$ for which the equation\begin{equation}
\label{tinyeq}
\boxed{\varphi(u)=C'}
\end{equation}has infinitely many roots which cluster down on the point $a$ as a limit point. \end{quote}

\end{theorem}

We remark that \textsc{Picard}'s ``big" theorem asserts if $C'$ be \emph{any} number chosen arbitrarily \emph {a priori}, with at most two excluded values, then  the equation\begin{equation}
\label{ }
\varphi(u)=C'
\end{equation}has infinitely many roots which accumulate down on $a$ as a limit point


\proof{(of the ``tiny" theorem)} By the classical \textsc{Weierstrass} theorem inside of a certain domain $|u-a|<\delta_1$ there is a point $u_1$ at which $\varphi(u)$ assumes a value $C_1$, which lies in the domain $|w-C|<h\equiv h_1$.  If  $\varphi'(u_1)\neq 0$, the equation\begin{equation}
\label{ }
w=\varphi(u)
\end{equation}defines a one-to-one mapping of the neighborhood of $u_1$ onto the neighborhood $T_1$ of $C_1$; we will take our first neighborhood to be a small circle, $|u-u_1|<\epsilon_1$, which, however, lies totally within the domain $|u-a|<\delta_1$, but does not extend out to the point $a$, itself, and for which, additionally, $T_1$ is totally within the domain $|w-C|<h_1$.  If it turns out that $\varphi'(u_1)= 0$, we can choose a second slightly altered point $u_1'$ for which $\varphi'(u_1')\neq 0$ while nevertheless $\varphi(u_1')$ lies within the domain $|w-C|<h_1$, and then we take this instead of the point $u_1$.

We now repeat this step, and take $C_1$ in place of $C$ and $h_2,\ \ \delta_2$ so small that the circle $|w-C_1|<h_2$ lies within $T_1$, while, on the other hand, no point of the circle $|u-a|<\delta_2$ has any points in common with the circle $|u-u_1|<\epsilon_1$.  This way we arrive at a new point, $u_2$ for which $C_2=\varphi(u_2)$ lies in the circle $|w-C_1|<h_2$ as well as for which $\varphi'(u_2)\neq 0$.  Therefore, a smaller circle $|u-u_2|<\epsilon_2$, which lies in the circle $|u-a|<\delta_2$ and which does not extend to $a$ will be mapped in a one-to-one manner onto a neighborhood $T_2$ of $C_2$; moreover we shall take $\epsilon_2$ so small that $T_2$ lies entirely inside of $T_1$.

By continually repeating this process we obtain an unbounded sequence of nested domains $T_1, T_2,\cdots$ which will be mapped in a one-to-one manner by means of the relation $$w=\varphi(u)$$onto the circles $|u-u_n|<\epsilon_n$, $n=1,2,\cdots$.  These (nested) domains have at least one point, $C'$, in common, and to this point corresponds an image point $u_n$ in each circle:\begin{equation*}
\label{ }
C'=\varphi(u_n)\ \ \ \ \ \ (n=1,2,\cdots)
\end{equation*}  The points $u_n$ form the sequence of roots  of  \eqref{tinyeq} which accumulate down on the essential singularity $u=a$. This completes the proof.\endproof

\begin{coro}\begin{quote}.

\emph{If $a=\infty$, the equation\begin{equation}
\label{tinyeq}
\boxed{\varphi(u)=C'}
\end{equation}has infinitely many roots which form a discrete set of isolated points in the plane. }\end{quote}\end{coro}

\begin{coro}\begin{quote}.

\emph{The theorem continues to hold if $\varphi(u)$ is meromorphic in the neighborhood of the essential singular point $u=a$.}\end{quote}\end{coro}
 
 \proof Indeed, let $C$ be an arbitrary number.   If $\varphi(u)$ takes the value $C$ in every prescribed neighborhood of the point $u=a$, then we are done.  Otherwise, the function \begin{equation}
\label{ }
\Phi(u)=\frac{1}{C-\varphi(u)}
\end{equation}where $\Phi(u)$ will be defined as $0$ at the poles of $\varphi(u)$ will have an essential singularity at the point $u=a$, just as assumed in the previous theorem, and the proof follows easily on the basis of the theorem.\endproof


\newpage
\subsection{Periodicity}

We now offer \textsc{Weierstrass}' beautiful proof of the following theorem, due to him:

\begin{theorem}

\begin{quote}.

A \textbf{meromorphic} function which admits an \textbf{ALGEBRAIC ADDITION THEOREM}:\begin{equation}
\label{AAT}
\boxed{G[\varphi(u+v),\varphi(u),\varphi(v)]=0}
\end{equation}is either:\begin{enumerate}
  \item a \textbf{rational} function of $u$; or
  \item a \textbf{periodic} function of $u$.
   \end{enumerate}\end{quote}\end{theorem}
   
   \proof If $u=\infty$ is not an essential singularity, then $\varphi(u)$ is a \emph{\textbf{rational}} function, and we are finished.
   
   Suppose, then, that $u=\infty$ \emph{is} an essential singularity of  $\varphi(u)$.  Let $m$ be the degree of the polynomial $G$ in the variable $\varphi(u+v)$.  By \textsc{Weierstrass}' ``tiny" \textsc{Picard} theorem \ref{tinyeq}, there exists a number $C_2$ for which the equation\begin{equation}
\label{ }
C_2=\varphi(v)
\end{equation}has at least $m+1$ roots, $v=a_0,a_1,\cdots,a_m$. We take $\upsilon$ such that $\varphi(u)$ is analytic at each of the points $\upsilon+a_i$, $i=0,1,\cdots,m$ as well as at the point $u=\upsilon$.  Put $C_1=\varphi(\upsilon)$.  We now form the quantities\begin{equation*}
\label{ }
\varphi(\upsilon+a_0),\ \ \varphi(\upsilon+a_1), \ \ \cdots,\ \ \varphi(\upsilon+a_m)
\end{equation*}Each of them is a root of the polynomial $G(W,C_1,C_2)$.  Therefore at least two of them must be equal to each other.

We will now show that for two of the above quantities $a_i$ the equation\begin{equation}
\label{A}
\varphi(u+a_k)=\varphi(u+a_l)
\end{equation}holds in general.  First, if we limit $u$ to the neighborhood of the point $\upsilon$ it follows from the equation\begin{equation*}
\label{ }
G[\varphi(u+a_i), \varphi(u), C_2]=0, \ \ \ \ \ \ (i=0,1,\cdots,m)
\end{equation*}that the equation\begin{equation*}
\label{ }
\varphi(u+a_{\kappa})=\varphi(u+a_{\lambda})
\end{equation*}most hold for every point of this same neighborhood.  Here, $\kappa$ and $\lambda$ might be different for different values of $u$.  However, since there are only a finite number of different pairs $(\kappa,\lambda)$ while the number of points $u$ is infinite, there must exist at least one pair $\kappa=k,\lambda=l$ for which the above equation holds for infinitely many points in the neighborhood of $\upsilon$, and therefore for all points of that neighborhood.  Therefore it follows that $\varphi(u+a_k)$ and $\varphi(u+a_l)$ are one and the same analytic function, and therefore, that $\varphi(u)$ has \emph{\textbf{the period}} $a_k-a_l$.\endproof

\begin{rem} The proof shows that the ``cause'' or 
``explanation'' of the existence of a period  of the 
meromorphic function~$\varphi$ is \textit{the simultaneous occurrence of 
two antithetical properties} of~$\varphi$:
\begin{itemize}
\item the wild chaotic dispersion of values of~$\varphi$ in the 
neighbourhood of the essential singularity at infinity (an 
``irresistible force'');
\item the rigid unyielding restriction on the values of $\varphi(u)$ 
imposed by the polynomial form of the AAT(an ``immovable object'').
\end{itemize}
The mathematical resolution of this ancient philosophical conundrum is 
this: \textit{the values of $\varphi(u)$ in the neighborhood of the essential singularity $u=\infty$ are compelled by the AAT to distribute themselves into classes of equal roots in which the values of the variables differ by a constant...the \emph{\textbf{period.}}}. This beautiful 
interpretation of periodicity is due to Weierstrass.\end{rem}
 
 
 \subsection{Forsyth's proof of periodicity}

In 1893, the british mathematician \textsc{Andrew Russell Forsyth}~\cite{Forsyth} published a textbook on analytic function theory.  Chapter XIII is devoted to proving Weierstrass' theorem.

His book has been heavily criticized from the day it was published because the author apparently had not mastered the subtleties of rigorous mathematics and his book is filled with inaccuracies and outright mistakes.  See Osgood's~\cite{Osgood2} scathing review.  

As Osgood points out, Forsyth's proof of Weierstrass' theorem is vitiated by his \emph{assumption} that Phragmen's theorem on the invariance of the addition theorem polynomial holds and that if the function $\varphi(u)$ is defined for $u$ and for $v$, then it is defined for $u+v.$

Nevertheless, it is worthwhile to sketch Forsyth's novel argument for proving the \emph{\textbf{periodicity}} of $\varphi(u)$ if $\varphi(u)$ admits the AAT in all of the complex plane and is not algebraic.  Just as in the argument of Weierstrass, the fundamental assumption is that if $U$ is a value of $\varphi(u)$,\emph{ then the solution set of the equation $\varphi(u)=U$ is an infinite discrete set.}  We also assume that $\varphi(u)$ is holomorphic in the neighborhood of $u=0$ and at the origin itself.

Therefore, the members of this set form one or more \emph{infinite sequences} whose typical $n^{th}$ member ($n=0,\pm1,\pm2,\cdots$) is denoted by $f(u,n)$ and where we assume $f(u,0)=u$.  The same is true for the solution set of the equation $\varphi(v)=V$ so that $f(v,m)$ is the solution set and 
$m=0,\pm1,\pm2,\cdots$, and for the value furnished by the addition theorem $\varphi(u+v)=W$ where $f(u+v,r)$    $r=0,\pm1,\pm2,\cdots$   is the solution set so that we obtain the \emph{functional equation}

\begin{equation}
\label{fe}
f(u,n)+f(v,m)=f(u+v,r)
\end{equation}
where we assume that $f$ is differentiable in the first variable.

 If we interchange $u$ and $v$ in \eqref{fe} we also interchange $n$ and $m$, which means that  the integer-valued function $r(n,m)$ is \emph{symmetric} in $n$ and $m$.
 
Moreover, 

\begin{equation}
\label{fe1}
\frac{\partial f(u,n)}{\partial u}=\frac{\partial f(u+v,r)}{\partial u+v}=\frac{\partial f(v,m)}{\partial v}\Rightarrow \frac{\partial^2}{\partial u}\frac{\partial f(u,n)}{\partial u}=0
\end{equation}

\noindent which means that $\frac{\partial f(u,n)}{\partial u}$  is a function of $u$ and $n$ which  is \emph{independent} of $u$, i.e., it is a function \emph{only} of $n$, say $\theta(n)$.  Integrating \eqref{fe1} we obtain

\begin{equation}
\label{for}
f(u,n)=u\theta(n)+\psi(n)
\end{equation}

\noindent where ``the constant of integration" $\psi(n)$ depends only on $n$.  Substituting this formula into the functional equation \eqref{fe} we obtain 

\begin{equation}
\label{fe2}
u\theta(n)+\psi(n)+v\theta(m)+\psi(m)=(u+v)\theta(r)+\psi(r)
\end{equation}

\noindent which holds for all values of $u$ and $v$.  Differentiating both sides with respect to $u$ gives us the equation $\theta(n)=\theta(r)$ while if we differentiate it with respect to $v$ we obtain $\theta(m)=\theta(r)$, i.e.,
$$\theta(n)=\theta(r)=\theta(m).$$This means that $\theta(n)$ is a \emph{constant} and is equal to its value when $n=0$.  Substituting this value in the formula \eqref{for} and remembering that $f(u,0)=u$, we obtain
$$f(u,0)=u=u\theta(0)+\psi(0)\Rightarrow 1=\theta(0)\equiv\theta(n)\ \ \ \text{and}\ \ \ \psi(0)=0.$$Therefore, the formula \eqref{for} takes the form

 \begin{equation}
\label{for2}
f(u,n)=u+\psi(n),
\end{equation}

\noindent where $\psi(0)=0.$ Taking $u=v=0$ in \eqref{fe2} we obtain the defining equation

\begin{equation}
\label{abel}
\psi(n)+\psi(m)=\psi(r)
\end{equation}

\noindent 
of the integer-valued function $\psi(n).$  The function
 
\begin{equation} 
\label{ }
r\equiv r(n,m)
\end{equation}has the properties \begin{enumerate}
  \item $r(n,0)=n$, $r(0,m)=m$
  \item $r(n,m)=r(m,n)$
  \item $r(n_1,r(n_2,n_3))=r(r(n_1,n_2),n_3)$
\end{enumerate}To find such a function we appeal to some elementary ideas from the theory of formal group laws and suppose that $r(n,m)$ is represented by a formal power series

\begin{equation}
\label{gl}
r(n,m)\equiv n+m+\sum c_{ij}n^im^j
\end{equation}

\noindent By conditions (i) and (ii) we can conclude that 

\begin{equation}
\label{gl1}
r(n,m)\equiv n+m+nm\cdot\sum b_{ij}n^im^j\equiv m+n+nm\lambda
\end{equation}

\noindent where $b_{ij}=b_{ji}.$  Since $\psi(0)=0$, we write 

$$\psi(n)\equiv n\chi(n)$$

\noindent and \eqref{abel} becomes

\begin{equation}
\label{abel2}
n\chi(n)+m\chi(m)=(n+m+nm\lambda)\chi(n+m+nm\lambda)
\end{equation}All of the mathematics developed here, starting with the functional equation \eqref{fe} is completely rigorous.  
\
\noindent However, at this point Forsyth writes 
\begin{quote}

``Since the left-hand side is the sum of two functions of distinct and independent magnitudes, the form of the equation shews that it can be satisfied only if $\lambda=0$, so that..."

\end{quote}

\noindent We do not understand what this sentence means.  We suppose that he argues that the right hand side of \eqref{abel2} must also be ``the sum of two functions of distince and independent magnitudes..." which certainly does occur if we take $\lambda=0,$ but this is not true in general as the equation $\ln x+\ln y =\ln (xy)$ immediately shows.

Moreover we have not been able to supply an independent argument that necessarily $\lambda=0.$  Indeed, it is not obvious that it is true.   But, if we \emph{assume} it to be the case we can conclude 
$$r=n+m\Rightarrow \chi(n)=\chi(r)=\chi(m)=\text{a constant} \equiv \omega\ \ \text{say}.$$Therefore, the final formula for $f(u,n)$ is \begin{equation}
\label{ }
\boxed{f(u,n)=u+n\omega}
\end{equation}That is to say, $\varphi(u)$\emph{is \textbf{periodic} with period} $\omega.$  Admittedly the argument is incomplete, but it is interesting and merits further study. (\footnote{The function $r(m,n)$ satisfies the famous \textsc{Abel} \emph{associativity functional equation}, the equation (iii), and it is well known that its general solution is of the form $r(m,n)= R^{-1}\{R(m)+R(n)\}$ for a suitable function $R$ and where $R^{-1}$ is the inverse function (see Aczel~\cite{Aczel}). Forsyth tries to prove that the function $R(m)=cm$ for some suitable constant $c$. \textsc{Gjergji Zaimi} and \textsc{Michael Somos} in Mathoverflow, question 288554, independently showed that if $\lambda$ is a \emph{constant} distinct from zero , then Forsyth's equation \eqref{abel2} has no nonzero solution.  Moreover their reasoning holds true if $\lambda$ is different from zero for all sufficiently large $\lambda$.  But there is no reason assume that.})
\\

 
 \section{Multiform functions having an AAT}
 
 In this section, the reader must know the elements of analytic continuation and the \textsc{Weierstrass} concept of a \emph{\textbf{complete analytic function}}.  We will preface our treatment of multiform functions which admit an AAT with some of the basic definitions and theorems of the \textsc{Weierstrass} theory.
 
 
 \subsection{Analytic Coninuation}
 
 We begin with the \textsc{Weierstrass} definition of a holomorphic function.

 \definition  \emph{A complex-valued function $f$ defined on an open subset $D\subset \mathbf{C}$ is called \emph{\textbf{holomorphic in D}} if each point $a\in D$ has an open neighborhood $U$, $a\in U\subset D$, such that the function $f$ has a power series expansion\begin{equation}
\label{holo}
\boxed{f(z)=\sum_{n=0}^{\infty}a_n(z-a)^n}
\end{equation}which converges for all $z\in U$.}

Usually $U$ is the open disc, $K(a)$ of radius $r(a)$ which is the interior of the circle of convergence of the power series\eqref{holo}.

\definition \emph{If $s$ is a point of $K(a)$, then $f(z)$ can be expanded in a power series \begin{equation}
\label{rearr}
\boxed{f(z)=\sum_{n=0}^{\infty}b_n(s)(z-s)^n}
\end{equation}of powers of $(z-s)$ which is obtained from \eqref{holo} by expanding the powers\begin{equation*}
\label{ }
(z-a)^n=[(z-s)+(s-a)]^n
\end{equation*}by the binomial theorem and arranging the series \eqref{holo} in powers of $(z-s)$.  We say that the series \eqref{rearr} arises from the series \eqref{holo} as a result of \textbf{rearranging} the series at $s$}

\definition \emph{If \begin{equation*}
\label{ }
r(s)>r(a)-|s-a|
\end{equation*}then the circle of convergence $K(s)$ of \eqref{rearr} extends beyond the circle $K(a)$, and \eqref{rearr} yields an \textbf{analytic continuation} of the function defined by \eqref{holo} in $K(a)$ into the part of $K(s)$ not covered by $K(a)$.}
\\
\\
The power series \eqref{holo} and \eqref{rearr} yield the \emph{\textbf{same}} functional values in the disc \begin{equation*}
\label{ }
|z-s|<r(a)-|s-a|.
\end{equation*}This is so because this disc is contained in the intersection $D$ of $K(a)$ and $K(s)$ and, since both power series converge in $D$, their values must coincide in $D$.

Let $b$ be a point of the $z$-plane and let $L$ be a curve joining $a$ to $b$.  $L$ is determined by a complex-valued function $z=\varphi(t)$ defined and continuous on a real interval $\alpha \leq t\leq \beta$.  If we divided the interval $\alpha \leq t\leq \beta$ into $n$ subintervals by means of an increasing sequence of points $\alpha:=\gamma_0, \gamma_1,\cdots, \gamma_n:=\beta$, then $L$ splits into $n$ subarcs $L_1,L_2,\cdots,L_n$ with $L_g\ \ (g=1,2,\cdots,n)$  joining $\varphi(\gamma_{g-1}):=c_{g-1}$ to $\varphi(\gamma_{g}):=c_{g}$.  Now let $K_0,K_1,\cdots,K_n$ be disks with centers $a:=c_0,c_1,\cdots,c_n:=b$, and let the closed arc $L_g$ lie completely in $K_{g-1}$ for $g=1,\cdots,n$.  Further let $f_0$ be the power series \eqref{holo} and for $g=1,\cdots,n$ let $f_g$ which results from rearrangement of $f_{g-1}$ at $c_g$ converge in $K_g$.  Then $f_0, f_1,\cdots, f_n$ have been recursively defined.

\definition\emph{We say that the power series $f_n(z)$ is the result of \textbf{analytic continuation of the power series $f_0(z)$ along the curve L.}}
\\
\\
It can be shown that this analytic continuation is independent of the choice of division points on $L$.

Now we state the fundamental definition.
 
\definition \emph{The different power series obtained by means of all possible analytic continuations of \eqref{holo} to arbitrary points in $D$ are called \emph{\textbf{ function elements}} or \emph{\textbf{single-valued branches}} of the \emph{\textbf{complete analytic function}} f determined by the power series \eqref{holo}}

This allows us to introduce the concept of a multiform function.

\definition \emph{A complete analytic function is \textbf{multiform} at a point if it has more than one analytic continuation to that point. Otherwise it is \textbf{uniform} or \textbf{single-valued}.}
\\
\\
A fundamental theorem in the theory of analytic continuation is the \textbf{permanence of the functional equation}.


 We shall consider a function $F(z,w_1,\cdots,w_n)$, analytic in each of the variables $z, w_1,\cdots,w_n$ in the regions $R_1,\cdots,R_n$, respectively.  Next we shall assume that there is a circular neighborhood $D_0:=\{z: |z-a|<r\subset R$ and $n$ analytic functions $f_k(z)$ in $D_0$ with function values in the respective regions $R_k$, i.e., such that $f_k(D_0)\subset R_k\ (k=1,\cdots, n)$ and such that the equation\begin{equation}
\label{fe}
F\{z,f_1(z),\cdots,f_n(z)\}=0
\end{equation}holds identically for $z\in D_0$.  In this situation we ay that  \eqref{fe} is a functional equation in the $w_k$, and the following theorem is called the \textbf{\emph{principle of permanence of functional equations}}, also known as \emph{\textbf{the invariance theorem for functional equations}}.

\begin{theorem} \begin{quote}.

Suppose that the functions $f_k(z)$ can be continued analytically along the rectifiable arc $\gamma: \ z=\psi(t),0\leq t\leq 1$ connecting the point $a\in R$ to the point $b\in R$.  Also, suppose that for each $t\in\left[0,1\right]$ there is a disk $D_t:=\{z: \ |z-\psi(t)|<r(t)\}$ such that $D_t\subset R$ and such that $f_{kt}(z)$ is analytic in $D_t$ with $f_{kt}(D_t)\subset R_k\ \ (k=1,\cdots,n)$.  Then the following equation \begin{equation}
\label{ }
F\{z,f_{1t}(z),\cdots,f_{nt}(z)\}=0
\end{equation}holds true for all $z\in D_t, 0\leq t\leq 1$. \end{quote}\end{theorem}

In other terms:  \emph{The analytic continuations of the solutions of a functional equation are solutions of the analytic continuation of the equation. } If a given element of an analytic function satisfies a certain functional equation, under the conditions of the permanence of the functional equations we apply the functional equation to obtain a continuation of the function to a larger region of analyticity, provided such a continuation exists.

 
 \subsection{Algebroid Functions}
 
 The multiform functions which interest us are the algebroid functions.
 
 \definition  \emph{A function $z=\varphi(u)$ defined by a particular element is \emph{\textbf{algebroid}}  in the interior of a given domain if, and only if, all of its values which may obtained by continuing the element along paths which totally belong to the given domain satisfy an equation of the form\begin{equation}
\label{algebroid}
p_0(u)z^n+p_1(u)z^{n-1}+\cdots+p_n(u)=0
\end{equation} where $p_0,\cdots,p_n$ are holomorphic in the interior of that same domain.}
\\
\\
One obtains, as a particular case, the \emph{\textbf{algebraic}} functions if the domain is the entire complex plane. 

 We will cite some properties of algebroid functions in the following.  Their proofs are word-for-word identical to the proofs for the corresponding properties of the algebraic functions.  The book by \textsc{Saks} and \textsc{Zygmund}~\cite{Saks} is a good reference for this material.

An algebroid function has two types of singular points in the interior of its domain: \textbf{\emph{poles}} and \textbf{\emph{branch points.}}  Moreover, a point may be both a pole and a branch point simultaneously.

The general theorem which characterizes algebroid functions is the following:

\begin{theorem}\begin{quote}.

Suppose that\begin{equation}
\label{FundThm}
F(u,z):=p_0(u)z^n+p_1(u)z^{n-1}+\cdots+p_n(u)
\end{equation}as a polynomial in $z$ is irreducible in the field of rational functions and that the $p_k(u)$ are holomorphic in a domain $D$.  Then the equation\begin{equation}
\label{ }
F(u,z)=0
\end{equation}defines a \textbf{complete analytic function} $z=\varphi(u)$ for all $u$ in $D$.  The singularities of $\varphi(u)$ come from three different sources:\begin{enumerate}
  \item The zeros of $p_0(u)$ are the \textbf{poles} of one or more of the branches of $\varphi(u)$.
  \item The roots of the equation in $u$ obtained by elimination of $z$ from the simultaneous equations\begin{equation}
\label{ }
F(u,z)=0, \ \ \ \frac{\partial F(u,z)}{\partial z}=0
\end{equation}may be \textbf{critical points} of some or all branches, but all branches tend to finite limits as $u$ approaches such a point.
  \item The point at infinity may be a pole or a critical point for some of the branches
\end{enumerate}At any other point, $\varphi(u)$ has $n$ holomorphic, distinct elements.  At the singular points the branches form \textbf{cycles} and there are one or more algebraic elements besides regular and polar elements.\end{quote}

 \end{theorem}

 \example The following example illustrates the different possibilities.  Let $z=\varphi(u)$ be defined by the equation \begin{equation}
\label{ }
8uz^3+3(1-u)z+1-u=0
\end{equation}Then the branch points are at $u=1$ (repeated) and at $u=-1$.  The only pole of $\varphi(u)$ is at $u=0$.

If $u=1$, then   the repeated value is $z=0$ occurring thrice and  one obtains the \emph{\textbf{algebraic element}}\begin{equation}
\label{ }
\varphi(u)=\frac{1}{2}(u-1)^{\frac{1}{3}}
\end{equation}and the three values are branches of one system of cyclical order for a circuit around $u=1$.

If $u=-1$, then either $\varphi_1=1$, which is an isolated non-repeated value, or the repeated value of $\varphi(u)$ is $\varphi_2(-1)=\varphi_3(-1)=-\frac{1}{2}$ occurring twice, and the expansions are\begin{eqnarray}
\varphi_1(u)&=&1+\frac{2}{9}(u+1)+\cdots\\
\varphi_2(u) & = & -\frac{1}{2}+\frac{1}{2\sqrt{6}}(u+1)^{\frac{1}{2}}+\cdots \\
\varphi_3(u) & = & -\frac{1}{2}-\frac{1}{2\sqrt{6}}(u+1)^{\frac{1}{2}}+\cdots 
\end{eqnarray}and the \emph{\textbf{algebraic elements}} $\varphi_2(u)$, $\varphi_3(u)$ are cyclically interchangeable for a small circuit round $u=-1$.

If $u=0$, it is a regular point for one branch $\varphi_3(u)$ of $\varphi(u)$ while it is a pole of order $-\frac{1}{2}$ for the other two branches, $\varphi_1(u)$ and $\varphi_2(u)$, and thus also a branch point for them, and they are cyclically interchangeable for a small circuit around $u=0$. The three expansions of the \emph{\textbf{holomorphic element}} $\varphi_3(u)$ and the \emph{\textbf{algebraic polar elements}} $\varphi_1(u),\varphi_2(u)$ are\begin{eqnarray}
\varphi_1(u) & = & \sqrt{\frac{3}{8}}iu^{-\frac{1}{2}}+\frac{1}{6}- \frac{17}{18}\sqrt{\frac{3}{8}}iu^{\frac{1}{2}}-\frac{4}{81}u-\frac{275}{1944}\sqrt{\frac{3}{8}}iu^{\frac{3}{2}}-\frac{4}{729}u^2+\cdots\\
\varphi_2(u) & = &  -\sqrt{\frac{3}{8}}iu^{-\frac{1}{2}}+\frac{1}{6}+ \frac{17}{18}\sqrt{\frac{3}{8}}iu^{\frac{1}{2}}-\frac{4}{81}u+\frac{275}{1944}\sqrt{\frac{3}{8}}iu^{\frac{3}{2}}-\frac{4}{729}u^2+\cdots\\ 
\varphi_3(u)&=&-\frac{1}{3}+\frac{8}{81}u+\frac{8}{729}u^2+\cdots
\end{eqnarray}

 \begin{coro}\begin{quote}.
 
 \emph{A single-valued algebroid function is \textbf{meromorphic}}.\end{quote}\end{coro}
 
 \begin{coro}\begin{quote}.
 
  \emph{A symmetric rational function of those branches of an algebroid function which form a cycle is \textbf{meromorphic}.}\end{quote}\end{coro}

 \subsection{\textsc{Weierstrass'} Lemma on Analytic Continuation}
 
 The following fundamental lemma is due to \textsc{Weierstrass} (reproduced by Falk~\cite{Falk} from Weierstrass' lectures in the winter semester 1885) 
 
 \begin{theorem}\begin{quote}.
 
 Let $\varphi(u)$ be an analytic function which is holomorphic at $u=0$.  Suppose one of its function elements $P(u)$ satisfies an \emph{\textbf{ALGEBRAIC EQUATION OF THE FORM}} \begin{equation}
\label{initial}
\boxed{f\left[P\left(\frac{u}{2}\right),P(u)\right]=0}
\end{equation}Then, for any arbitrarily large, but \emph{\textbf{FINITE}} number $R$, the function \emph{\textbf{exists}} in the \emph{\textbf{entire}} domain \begin{equation}
\label{}
|u|<R
\end{equation}and is \emph{\textbf{ALGEBROID}} there.
  \end{quote}
 \end{theorem}
 
 \begin{proof}By assumption, the element $P(u)$ converges in some domain\begin{equation}
\label{ }
|u|<\rho
\end{equation}and satisfies the equation\begin{equation}
\label{u2}
\boxed{f\left[P\left(\frac{u}{2}\right),P(u)\right]=0}
\end{equation}in that domain.  Moreover, we can assume that the polynomial $f$ in \eqref{initial} is irreducible.

Now, if, in \eqref{initial}, we replace $u$ by $\dfrac{u}{2}$, then in the result we again replace $u$ by $\dfrac{u}{2}$, and if we carry out this process an arbitrary number of times, say $m>0$, where $m$ is an arbitrary integer, we obtain\begin{eqnarray}
f\left[P\left(\frac{u}{2}\right),P(u)\right] & = & 0\ \ \ \text{for}\ \ |u|<\rho \\
f\left[P\left(\frac{u}{2^2}\right),P(\frac{u}{2})\right] & = & 0\ \ \ \text{for}\ \ |u|<2\rho \\
f\left[P\left(\frac{u}{2^3}\right),P(\frac{u}{2^2})\right] & = & 0\ \ \ \text{for}\ \ |u|<2^2\rho \\
\cdots&\cdots&\cdots\\
f\left[P\left(\frac{u}{2^m}\right),P(\frac{u}{2^{m-1}})\right] & = & 0\ \ \ \text{for}\ \ |u|<2^{m-1}\rho 
\end{eqnarray}For the moment we put\begin{equation}
\label{xs}
x:=P(u), \ \ x_1:=P\left(\frac{u}{2}\right), \ \ \cdots,\ \ x_m:=P\left(\frac{u}{2^{m}}\right)
\end{equation}
Then we can write the above equations in the simple form:\begin{equation}
 \label{array}
 f(x_1,x)=0, \ \ f(x_2,x_1)=0, \cdots, \ \ f(x_m, x_{m-1})=0
 \end{equation}Since the number of equations in \eqref{array} is $m$, we can eliminate the $m-1$ quantities\begin{equation}
\label{ }
x_1,\ \ x_2,\cdots,\ \ x_{m-1}
\end{equation}from them to obtain an equation\begin{equation}
\label{one}
\gamma(x_m,x)=0
\end{equation}where the left hand side is a polynomial in $x_m$ and $x$, and this equation is valid whenever those in \eqref{array} are.

Therefore, we have found that within$$|u|<\rho$$ the equation \eqref{one} is satisfied by \begin{equation}
\label{ }
x=P(u), \ \ x_m=P\left(\frac{u}{2^{m}}\right)
\end{equation}because \eqref{one} arises when the quantities $x_1,\cdots,x_m$ are defined by \eqref{xs}.

Now, \eqref{one} has the form\begin{equation}
\label{ }
\gamma_0(x_m)x^n+\gamma_1(x_m)x^{n-1}+\cdots+\gamma_n(x_m)=0
\end{equation}where the $\gamma_j(x_m)$ are polynomials in $x_m$.  Therefore, since\begin{equation}
\label{ }
\gamma_j(x_m)=\gamma_j\left[P\left(\frac{u}{2^{m}}\right)\right]=:p_j(u)
\end{equation}where $p_j(u)$ is an ordinary power series in $u$ which converges throughout the domain 
\begin{equation}
\label{bigdomain}
|u|<2^m\rho
\end{equation}we find that the equation\begin{equation}
\label{Algebroid}
p_0(u)z^n+p_1(u)z^{n-1}+\cdots+p_n(u)=0
\end{equation}is satisfied in the smaller domain $$|u|<\rho$$ by \begin{equation}
\label{Pu}
x=P(u).
\end{equation}But, the integer $m$ can be chosen as large as we wish, so it is always possible to select $m$ large enough to ensure that $$2^m\rho>R$$ for any fixed $R$.  Hence, \emph{the power series $p_i(u)$ converge throughout the domain}\begin{equation}
\label{R}
|u|\leq R
\end{equation}.

Now, if the equation \eqref{Algebroid} is reducible in the domain \eqref{R}, we can factor it into irreducible equations, and one of these factors will then be satisfied in $|u|<\rho$ by \eqref{Pu}.  Hence, we may use this latter equation instead of \eqref{Algebroid}, or what is the same thing, assume that \eqref{Algebroid} is \emph{\textbf{irreducible}} in the domain \eqref{bigdomain}.

By the fundamental theorem, \ref{FundThm}, it follows that equation \eqref{Algebroid} yields $n$ single-valued brandhes of an analytic function of $u$ (each branch single-valued in \eqref{R}), and these branches form a single sycle in the domain \eqref{bigdomain}.  Moreover, since equation \eqref{Algebroid} (as proven above) is satisfied by \eqref{Pu} in the smaller domain $|u|<\rho$, we conclude that in its domain of convergence, $P(u)$ is an element of one of these $n$ single-valued branches, and hence, an $n$-valued multiform function, $\varphi(u)$ arises within the domain \eqref{bigdomain} from the element $P(u)$.  This function therefore exists in the domain \eqref{R}, and arises from the equation \eqref{u2}, valid in the domain $$|u|<\rho.$$

Then by the permanence of the functional equation, the function $\varphi(u)$ satisfies equation \begin{equation}
\label{ }
f\left\{\varphi\left(\frac{u}{2}\right), \varphi(u)\right\}=0
\end{equation}everywhere in the finite plane.\end{proof}


\subsection{Statement  of the general problem}
 
 The general problem which stands at the beginning of the whole theory can be formulated as follows: (We follow Mittag-Leffler~\cite{Mittag})
 \\

 \emph{\textbf{PROBLEM:}}

\begin{quote}\emph{Within the domain of regularity of the analytic function $\varphi(u)$ there exist three points
\begin{equation}
\label{ }
u=a, \ \ v=b, \ \ u+v=a+b
\end{equation}such that if
\begin{equation}
\label{elements}
P_1(u|a), \ \ P_2(v|b), \ \ P_3(u+v|a+b)
\end{equation}denote the respective elements of 
\begin{equation}
\label{ }
\varphi(u)\, \ \ \varphi(v), \ \ \varphi(u+v)
\end{equation}in the corresponding domains of convergence of the form
\begin{equation}
\label{nb}
|u-a|<\rho_1, \ \ |v-b|<\rho_2, \ \ |u+v-a-b|<\rho_3
\end{equation}then an \textbf{\emph{algebraic equation}}
\begin{equation}
\label{CP}
G(x,y;z)=0
\end{equation}is satisfied for
\begin{equation}
\label{ }
x=P_1(u|a), \ \ y=P_2(v|b), \ \ z=P_3(u+v|a+b)
\end{equation}as long as the three power series converge.
}\end{quote}

\begin{quote}\textbf{\emph{Which elements have his characteristic property and to what functions to they belong?}}\end{quote}

The problem is solved by the following steps:\begin{enumerate}
  \item we prove that the function $\varphi(u)$ is algebroid in any \shout{bounded} domain.
    \item we prove that $\varphi(u)$ has a \shout{finite number of branches} in all the complex plane.
  \item we prove Theorem \ref{2}.
\end{enumerate}


\subsection{$\varphi(u)$ is algebroid in any bounded domain}

There are two different proofs available for this.  The first proof is by \textsc{Phragm\'en} and the second is due, in principle, to \textsc{Weierstrass}.  The latter's proof has the advantage that it is applicable to the more general case of the \shout{abelian} functions of several complex variables.


\subsubsection{\textsc{Phragm\'en}'s proof}

W suppose that the equation \begin{equation}
\label{ }
G[P_1(u|a),P_2(v|b), P(u+v|a+b)]=0
\end{equation}holds between three elemnts $P_1,P_2,P$ of the function $\varphi$.

We start with the element $P_1(u|a)$, say.  Then there exists a positive quantity, say $r_1$, such that the function defined by this element is algebroid within the domain $$|u-a|<r_1.$$This is most certainly true if we choose $r_1$ smaller than the radius of convergence of the element $P_1(u|a)$.  There exist analogous numbers $r_2$ and $r$ for the elements $P_2(v|b)$ and $ P(u+v|a+b)$, respectively.
\\

\shout{CLAIM}:

\begin{quote} \shout{The upper limits of the quantities $r_1$, $r_2$, and $r$ are all finite, or they are all infinite at the same time.}

\end{quote}
 
 \begin{proof}  Suppose that one of the upper limits be infinite, that is, suppose that the function $\varphi(u)$ is algebroid in every finite domain.  Then it is necessary that the others also have to be infinite, too.  
 
 Now we will show that all three cannot be finite.  For, if this be the case, let $\rho_1$, $\rho_2$, and $\rho$ be these finite limits.  Then one of the following two inequalities must hold:\begin{equation}
\label{ }
\rho_1<\rho+\rho_2 \ \ \ \ \ \text{or}\ \ \ \ \ \rho_2<\rho+\rho_1
\end{equation}Define\begin{equation}
\label{ }
v-b:=-\frac{\rho_2}{\rho+\rho_2}(u-a).
\end{equation}Then we obtain that \begin{equation}
\label{ }
u+v-a-b=\frac{\rho}{\rho+\rho_2}(u-a).
\end{equation}Suppose, now, that \begin{equation}
\label{ }
|u-a|<R<\rho+\rho_2.
\end{equation}Then we obtain\begin{equation}
\label{ }
|v-b|<\frac{R}{\rho+\rho_2}\rho_2<\rho_2
\end{equation}and \begin{equation}
\label{ }
|u+v-a-b|<\frac{R}{\rho+\rho_2}\rho<\rho.
\end{equation}This means that the functions of $u$ defined by the elements $P_2$ and $P$ after this substitution are algebroid in all of the domain \begin{equation}
\label{ }
|u-a|<R
\end{equation}where $R$ is a positive quantity smaller than $\rho+\rho_2$.  Let $y$ and $z$ denote these functions, and let $x$ be the function defined by the element $P_1$.  Then, by assumption, the following equation\begin{equation}
\label{ }
G(x,y,z)=0
\end{equation}holds.  If we eiminate $y$ and $z$ between this equation and the two equations\begin{eqnarray}
y^{\mu}+ \phi_1y^{\mu-1}+\cdots+\phi_{\mu}& = & 0 \\
z^{\nu}+ \psi_1y^{\nu-1}+\cdots+\psi_{\nu}& = & 0
\end{eqnarray}which define $y$ and $z$ as algebroid functions in the interior of the domain $|u-a|<R$, we obtain the resultant equation\begin{equation}
\label{ }
z^{n}+ f_1y^{n-1}+\cdots+f_{n}=  0
\end{equation}where $f_1,\cdots,f_n$ are meromorphic in the domain $|u-a|<R$.

This means that $x$ is algebroid in the interior of this same domain $|u-a|<R$ where $R$ is any positive quantity smaller than $\rho+\rho_2$.  But, if we take $R$ to satisfy $\rho_1<R<\rho+\rho_2$, we \shout{contradict the assumption of the maximality of} $\rho_1$.

The same reasoning applies to the case $\rho_2<\rho+\rho_1$.

 Therefore we have proven
 
 \begin{theorem}\begin{quote}.
 
 Suppose an analytic function $\varphi(u)$ admits an \shout{algebraic addition theorem}.  Then:\begin{enumerate}
  \item it exists in every finite domain $$|u|<R.$$
  \item it is \shout{algebroid everywhere} in that domain.
 \end{enumerate}
 \end{quote}\end{theorem}
  \end{proof}
 
 \subsubsection{\textsc{Weierstrass}'s proof}

\textsc{Weierstrass}'s proof is based on his fundamental continuation lema (theorem 7).  However, in order to apply it, we must be sure we comply with the hypothesis that the function has an element which is holomorphic at the origin.

\begin{theorem}
\begin{quote}.

The analytic function $\varphi(u)$ has an algebraic addition theorem
\begin{equation}
\label{ }
G\{P_1(u-a, P_2(v-b), P_3(u+v-a-b)\}=0
\end{equation}if and only if the function, $\psi(x)$, defined by \begin{equation}
\label{ }
\psi(x):=\varphi(x+a)
\end{equation}has the properties:\begin{enumerate}
  \item $\psi(x)$ possesses a functional element which exists in a neighborhood of $x=0$, call it $\psi_0(x)$; and
  \item between \shout{its} functional values\begin{equation}
\label{ }
\psi_0(x), \ \ \psi_0(y), \ \ \psi_0(x+y)
\end{equation}where $y=v-b$ there exists an algebraic equation\begin{equation}
\label{ }
\bar{G}\{\psi_0(x),\psi_0(y),\psi_0(x+y)\}=0
\end{equation}with coefficients independent of $x$ and $y$.
   
\end{enumerate}

\end{quote}
\end{theorem}
 
 \begin{proof} (\textsc{K\"obe~\cite{Koebe}})
 First, assume that $\varphi(u)$ admits an algebraic addition theorem.  Then we suppose that in the neighborhood of each of the respective points $$u=a,\ \ v=b,\ \ u+v=a+b$$the function $\varphi(u)$ is represented by the respective power series \begin{equation}
\label{ }
P_1(u-a), \ \ P_2(v-b), \ \ P_3(u+v-a-b)
\end{equation}and that there exists an algebraic relation\begin{equation}
\label{AT}
G\left[P_1(u-a), P_2(v-b), P_3(u+v-a-b)\right]=0
\end{equation}holding in the common domain of $P_1,P_2,P_3$.  Set\begin{equation}
\label{ }
x:=u-a, \ \ y:=v-b, \ \ x+y:=u+v-a-b.
\end{equation}Then \eqref{AT} becomes\begin{equation}
\label{AT1}
G\left[P_1(x), P_2(y), P_3(x+y)\right]=0
\end{equation}Since $\varphi(x+a)=\psi(x)$ is holomorphic at $u=a$, $P_1(x)$ is holomorphic at $x=0$, and this proves part 1.  

In \eqref{AT1} we put $x=0$ and then $y=0$, which gives us the pair of equations\begin{eqnarray}
G_1\left[P_2(y),P_3(y)\right] & = & 0 \\
G_2\left[P_1(x),P_3(x)\right] & = & 0 
\end{eqnarray}Since these are algebraic identities, we can replace $x$ by $y$ in the second equation to obtain\begin{equation}
\label{G2}
G_2\left[P_1(y),P_3(y)\right]  =  0.
\end{equation}Now we eliminate $P_3(y)$ from the first equation in the above pair as well as from 
\eqref{G2} to get \begin{equation}
\label{G3}
G_3\left[P_1(y),P_2(y)\right]=0.
\end{equation}In \eqref{G2} we replace $y$ by $x+y$ to get\begin{equation}
\label{G2+}
G_2\left[P_1(x+y),P_3(x+y)\right]  =  0
\end{equation}Eliminate $P_3(x+y)$ between \eqref{AT1} and \eqref{G2+} to get\begin{equation}
\label{G4}
G_4\left[P_1(x),P_2(y),P_1(x+y)\right]=0
\end{equation}Finally, eliminate $P_2(y)$ between \eqref{G4} and \eqref{G3} to get\begin{equation}
\label{ }
\bar{G}\left[P_1(x),P_1(y),P_1(x+y)\right]=0
\end{equation}holding in some neighborhood of the origin.  If we take\begin{equation}
\label{ }
\psi_0(x):=P_1(x)
\end{equation} for some $a$ close to $x=0$, we have proved property 2.

Now suppose that the function $\psi(x)$ is defined by the equation \begin{equation}
\label{ }
\psi(x):=\varphi(x+a)
\end{equation}and has the properties 1 and 2.  Then, since $\psi(x)$ is holomorphic at $x=0$, the function $\varphi(u)$ is holomorphic about $u=a$, and has the expansion $P_1(u-a)$ there, i.e., $$\psi_0(x)=P_1(u-a).$$Similarly, for a suitable $b$, we get   $$\psi_0(y)=P_1(v-b),$$ whence by property 2 , \begin{equation}
\label{ }
\bar{G}\{P_1(u-a),P_2(v-b),P_3(u+v-a-b)\}=0
\end{equation}i.e., $\varphi(u)$ admits an algebraic addition theorem.

\end{proof}

Now we can give \textsc{Weierstrass}' proof (as reproduced by Falk~\cite{Falk} from the winter semester of 1885) that $\varphi(u)$ is algebroid in every finite domain.

\begin{proof}By the previous theorem, we can assume that $\varphi(u)$ is holomorphic at $u=0$.  Suppose that $u$ and $v$ satisfy the conditions\begin{equation}
\label{ }
|u|<\frac{\rho}{2}, \ \ \ |v|<\frac{\rho}{2}
\end{equation}where $\rho$ is the radius of convergence of that power series $P(w)$ which represents the single-valued branch $\varphi(u)$ in the form\begin{equation}
\label{ }
\varphi(w)=P(w)
\end{equation}Then, by the previous theorem \begin{equation}
\label{ }
G\left[P(u),P(v),P(u+v)\right]=0
\end{equation}holds for \begin{equation}
\label{ }
|u|<\frac{\rho}{2}, \ \ \ |v|<\frac{\rho}{2}
\end{equation}Hence, if we put $u=v$, we obtain a result of the form \begin{equation}
\label{ }
f\left[P(u),P(2u)\right]=0\ \ \ \ \text{for}\ \ |u|<\frac{\rho}{2}
\end{equation}or \begin{equation}
\label{ }
f\left[P(\frac{u}{2}),P(u)\right]=0\ \ \ \ \text{for}\ \ |u|<\rho,
\end{equation}where $f(z,x)$ denotes a polynomial in $z$ and $x$.  The theorem now follows from \textsc{Weierstrass}' continuation lemma.\end{proof}


\subsection{\textsc{K\"obe}'s proof of Weierstrass' theorem}

\textsc{Paul K\"obe} was a German mathematician who is famous for his proof of the general uniformization theorem for algebraic functions.  He wrote his doctor's thesis under the direction of \textsc{H.A. Schwarz}, one of \textsc{Weierstrass}' most devoted students, and the man who replaced him in Berlin, when \textsc{Weierstrass} retired.  \textsc{K\"obe}'s thesis is entitled \emph{``\"Uber diejenigen analytischen Funktionn eins Arguments, welche ein algebraisches Additionstheorem besitzen"}~\cite{Koebe} contains a proof of the solution of the problem for multiform functions.  His fundamental step is:

\begin{theorem}
\begin{quote}.

Suppose the analytic function $\varphi(u)$ admits an algebraic addition theorem.  Then, the number of branches of the function $\varphi(u)$ is \shout{finite}, and every branch is an \shout{algebraic function} of every other branch.
\end{quote}
\end{theorem}
 
 \begin{proof}

We know that $\varphi(u)$ has an element, $\varphi_0(u)$, which is holomorphic in the neighyborhood of the point $u=0$, and whose function values $\varphi_0(u)$, $\varphi_0(v)$ and $\varphi_0(u+v)$ satisfy an algebraic equation with coefficients independent of $u$ and $v$:\begin{equation}
\label{aat0}
G\left[\varphi_0(u),\varphi_0(v),\varphi_0(u+v)\right]=0
\end{equation}Let $r$ be the radius of convergence of the power series in $u$, which represents the holomorphic element $\varphi_0(u)$.  Following the proof of \textsc{Weierstrass}' continuation lemma, we put $u=v$ in \eqref{aat0} and obtain the algebraic equation \begin{equation}
\label{ }
G\left[\varphi_0(u),\varphi_0(u)\varphi_0(2u)\right]\equiv H_1\left[\varphi_0(u),\varphi_0(2u)\right]=0
\end{equation}.

  If $G\left[\varphi_0(u),\varphi_0(v),\varphi_0(u+v)\right]$ is an irreducible plynomial in the quantities $\varphi_0(u)$, $\varphi_0(v)$ and $\varphi_0(u+v)$, which may always be assumed, then it \shout{cannot} happen that all of the coefficients of the polynomial $H_1\left[\varphi_0(u),\varphi_0(2u)\right]$ would vanish.  Otherwise the polynomial $G$ would be \shout{divisible} by $\varphi_0(u)-\varphi_0(v)$.
  
  Then, as in the proof of the \textsc{Weierstrass} lemma, there exists an algebraic equation with constant coefficients between $\varphi_0(u)$, and $\varphi_0(\frac{u}{2^n})$:\begin{equation}
\label{ }
H_n\left[\varphi_0\left(\frac{u}{2^n}\right),\varphi_0(u)\right]=0
\end{equation}which holds for $|u|<r$.  

In this equation, $\varphi_0(\frac{u}{2^n})$ is a meromorphic function in the domain \begin{equation}
\label{Hn}
|u|<2^nr
\end{equation}Therefore, equation \eqref{Hn} shows that starting with the holomorphic element , $\varphi_0(u)$, of $\varphi(u)$, it can be analytically continued throughout the interior of a circle, $K_n$, centered at the origin, with radius $2^nr.$

By means of this continuation the function $\varphi(u)$ will  have only a finite number of branches within $K_n$ and the function will be algebroid at each point $u_0$ of the circle $K_n$. 

If $\varphi_1(u)$ and $\varphi_2(u)$ denote any two of these branches, then they satisfy the equations \begin{equation}
\label{ }
H_n\left[\varphi_0\left(\frac{u}{2^n}\right),\varphi_1(u)\right]=0, \ \ \ H_n\left[\varphi_0\left(\frac{u}{2^n}\right),\varphi_2(u)\right]=0
\end{equation}from which the uniquely determined quantity $\varphi_0\left(\frac{u}{2^n}\right)$may be eliminated.  The result of this elimination is an algebraic equation\begin{equation}
\label{ }
A\left[\varphi_1(u), \varphi_2(u)\right]=0
\end{equation}i.e., \shout{each branch is an algebraic function of the other.}

Since the number of singular points of the function $\varphi(u)$ within the circle $K_n$ is finite, there exist infinitely many pairs of mutually perpendicular diameters of this circle which do not go through any singular points (different from $u=0$).  We choose a definite one of these pairs.  Let the points of one of these diameters geometrically represent the values of the quantity $u_1$ and the points of the other diameter the values of the quantity $u_2$.  With this definition of $u_1$ and $u_2$, every quantity $u$ whose absolute value is smaller than $2^nr$ can be uniquely represented in the form\begin{equation}
\label{ }
u=u_1+u_2.
\end{equation}

For all values of $u$, whose absolute value is smaller than $r$, there is a uniquely determined branch of the function $\varphi(u)$, namely $\varphi_0(u)$.  The analytic continuations of this branch along the diameters within the circle $K_n$ are unique because, by assumption, neither diameter goes through a singular point.   We will denote these uniquely determined continuations of the branch $\varphi_0(u)$ by $\varphi_0(u_1)$ and $\varphi_0(u_2)$.

Now, let $u$ describe \shout{an arbitrary curve} $L$ from the origin to an interior point $u'=u'_1+u'_2$ of the circle $K_n$ which lies entirely within the given circle and which goes through no singular point.  As long as the absolute value of $u$ does not exceed, $r$, according to the fundamental assumption, the equation\begin{equation}
\label{ }
G\left[\varphi_0(u_1),\varphi_0(u_2),\varphi_0(u)\right]=0
\end{equation}By the permanence of the functional equation, this equation remains true for all analytic continuations, and therefore, if $\varphi_0(u)$ transforms into $\varphi_0(u')$ while being continued along the path $L$, whence $\varphi_0(u_1)$ and $\varphi_0(u_2)$ transform, respectively into $\varphi_0(u'_1)$ and $\varphi_0(u'_2)$, the following equation holds\begin{equation}
\label{ }
G\left[\varphi_0(u'_1),\varphi_0(u'_2),\varphi_0(u')\right]=0.
\end{equation}

Now comes the fundamental observation: in this equation $\varphi_0(u'_1)$ and $\varphi_0(u'_2)$
\shout{are independent} of the particular path joining the origin to $u'$. Therefore, the number $h_n$ of branches which can result from analytic continuation of the function $\phi(u)$ within the circle $K_n$ cannot exceed the degree, $g$, which is independent of $n$, of the fundamental equation \begin{equation*}
\label{ }
G[\phi_0(u),\phi_0(v),\phi_0(u+v)]=0
\end{equation*}with respect to 
$\phi_0(u+v)$.  This observation holds for every $n$ however large.  Therefore, the  positive numbers $h_1, h_2, h_3, \cdots$ satisfy the inequalities\begin{equation*}
\label{ }
h_1\leq h_2\leq \cdots \leq h_n\leq h_{n+1}\leq \cdots \leq g
\end{equation*}from which we may conclude that there exists a positive whole number $N$ with the property that  \begin{equation*}
\label{ }
h_N=h_{N+1}=h_{N+2}=\cdots \text{to infinity}
\end{equation*}that is, the set of all the $h_N$ branches of the function $\phi(u)$ obtained from the element $\phi_0(u)$ by analytic continuation in the interior of the circle $K_N$ represents the totality of all branches of the function $\phi(u)$.  For the transition to further circles of greater radius the analytic continuation of the function $\phi(u)$ results in a simultaneous continuation of the $h_N$ branches which transform into \emph{each other} when they circulate around any newly occurring branch point.

It was shown earlier that any two branches of the analytically continued function $\phi(u)$ within the circle $K_n$ are connected by an algebraic equation with constant coefficients.  We can now assert that this theorem holds everywhere for the function $\phi(u)$.  The proof results immediately if we replace $n$ by $N$.

The previous considerations have proven the following: 

\emph{The domain of the argument of the function $\phi(u)$ can be extended to all finite values without the function ceasing to be algebroid.  The number of branches of the function $\phi(u)$ is finite, and every branch is an algebraic function of every other branch. }
\end{proof}

Now we complete Koebe's proof of Weierstrass' theorem.

\begin{proof} [Proof of Weierstrass' theorem]\

Among the elementary symmetric functions of all of the branches of the function $\phi(u)$ there exist at least \emph{one} function $\psi(u)$ which is not a constant.  The function $\psi(u)$ is a single-valued function of the argument $u$ and thus  a \emph{meromorphic} function for all finite values of $u$.  For $|u|<r$, $\psi(u)$ is an algebraic function of $\phi_0(u)$ since all branches of the function $\phi(u)$ are algbraic functions of the branch $\phi_0(u)$ for $|u|<r$.  Therefore, the single-valued function $\psi(u)$ also has an \emph{algebraic addition theorem}, and the \emph{function $\phi(u)$} can be considered \emph{as an algbraic function of the single-valued function $\psi(u)$.} 

This completes \textsc{Koebe}'s proof of \textsc{Weierstrass}' theorem.
\end{proof}


 
 \subsection{\textsc{Phragmen}'s proof of Weierstrass' theorem}
 
 \textsc{Phragmen}~\cite{Phragmen} was the first mathematician to publish a proof of Weierstrass' theorem.  He gives a direct proof of the following theorem, which, in turn gives as an immediate corollary that  the number of branches of $\phi(u)$ is finite.

 \begin{theorem}
 \begin{quote}.
 
  We always have the relation \begin{equation*}
\label{ }
G\left[\phi(u),\phi(v),\phi(u+v)\right]=0
\end{equation*} between three values of the function $\phi$ corresponding to the values $u,v,u+v$ of the argument, i.e.,\textbf{ the same polynomial} connects any three values of $\phi(u),\phi(v),\phi(u+v)$ regardless of what branches we substitute.
\end{quote}
\end{theorem}

\begin{proof}

We will prove this by showing that each value $\phi_{a}$, which the function $\phi(u)$ can take at the point  $a$, which $\phi(u)$ approaches indefinitely when  $u$ approaches $a$ indefinitely along a certain path, satisfies the equation\begin{equation*}\label{AAT}
G\left[\phi_a,\\ \mathrm{P_2(v|b)}, \\ \\ \mathrm{P(u+v|a+b)}\right]=0.
\end{equation*}

In fact, one arrives at the vaue $\phi_a$ by continuing the element $\mathrm{P_1(u|a)}$ along a certain closed path.  One may always chose a continuous, finite and simply connected domain which contains this path completely.  One may then determine a second continuous and finite domain such that the lower limit of the distances of an interior point of the first domain and a point in the exterior of the second are larger than a quantity $d>|b|$.
We imagine, for a moment, a third domain large enough to contain in its interior not only the last of the two domains we have just fixed, and its boundary, but also the path along which it is necessary to continue the element $\mathrm{P_1(u|a)}$ in order to arrive at the element $\mathrm{P(u+v|a+b)}$.  In the interior of this domain, the function defined by the element $\mathrm{P_1(u|a)}$--or, which gives the same function, by $\mathrm{P(u+v|a+b)}$--is algebroid, and therefore it has a finite number of singular points in the interior of the larger of the first two domains.  We may therefore choose two points, $a'$, $b'$ in the neighborhood of $a$ and $b$ such that they satisfy the following conditions: in the first place it is necessary that we have $|b'|<d$,   in order that to each point $\alpha$ in the interior of the smaller of the two domains corresponds a point $\alpha+b'$ in the interior of the larger.  Moreover, the point $a'$ will be chosen in the interior of the small domain in such a way that the points $a'$, $a'+b'$ do not coincide with any of th singular points, of which there are a finite number in the large domain which constitute the singular points of the function considered up to now.  Finally, if $\alpha$ is a singular point of this function situated in the interior of the small domain, $\alpha+b'$ will be a regular point of the same function.

Without changing the final value $\phi_a$, we may now replace the closed path $a\cdots a$ by a set of regular paths as follows: $1^{\text{o}}$ a path $aa'$, $2^{\text{o}}$ a number of \emph{loops}\footnote{By a \emph{loop} joining the point $a'$ to a singular point $a$, we mean a regular path composed of a path leaving $a'$ and arriving at the neighborhood of $a$, a small circle around $a$, and the first path described in the opposite direction.  For exampe, see ~\cite{Forsyth}} joining the point $a'$ to the singular points situated in the smaller domain, and $3^{\text{o}}$ the path $a'a$.  We may choose these loops in such a way that the corresponding loops which the point $u+b'$ describes are composed of regular paths and of circles which do not contain any singular point.  Therefore, if one continues the system of three elements $\mathrm{P_1(u|a)},\mathrm{P_2(v|b)},  \mathrm{P(u+v|a+b)}$ by varying $1^{\text{o}}$ $u$ from $a$ to $a'$ and $v$ from $b$ to $b'$, $2^{\text{o}}$ $u$ from $a'$ to $a'$ along the loops, and $3^{\text{o}}$ $u$ from $a'$ to $a$ and $v$ from $b'$ to $b$, by starting with the element $\mathrm{P_1(u|a)}$ one approaches indefinitely the value $\phi_a$, while the elements $\mathrm{P_2(v|b)}$ and $\mathrm{P(u+v|a+b)}$ \emph{return to themselves}.

But, it has therefore been demonstrated that every regular or singular element $\phi(u|a)$ of the function $\phi(u)$ in the neighborhood of the point $a$ necessarily satisfies \emph{the same equation}\begin{equation*}\label{AAT}
G\left[\phi(u|a),\mathrm{P_2(v|b)},  \mathrm{P(u+v|a+b)}\right]=0.
\end{equation*} Since the analogous result holds for the elements $\mathrm{P_2}$ and $\mathrm{P}$, and since moreover the points $a,b$ were taken totally arbitrarily, our assertion is completely justified.

From this fact that the relation\begin{equation*}
\label{ }
G\left[\phi(u),\phi(v),\phi(u+v)\right]=0
\end{equation*}always holds it follows that the function $\phi(u)$ \emph{has only a finite number of values at any point $$\phi_1(u),\phi_2(u), \cdots, \phi_n(u)$$ which does not exceed the degree of the polynomial $G$ in the variable $\phi(u+v)$.
}
And now the proof proceeds as in the proof of \textsc{Koebe}. 
 
\end{proof}
 
  \subsection{\textsc{Schwarz}'s proof of Weierstrass' theorem}
  
  In the final chapter of his book on elliptic functions, \textsc{Harris Hancock}~\cite{Hancock} exposits a proof which he attributes to \textsc{H.A. Schwarz}.  It is different from the proofs of \textsc{Koebe} and \textsc{Phragmen} in that it avoids the necessity of proving that the number of branches is finite nor does it appeal to the theory of symetric functions of the roots.  
  
But, just like the two previous proofs, it appeals to the fact that \emph{the singularities of $\phi(u)$ form a finite discrete set of points in any bounded set in the plane.}

First \textsc{Schwarz} proves a very interesting algorithmic lemma.
\begin{lemma}
\begin{quote}.

There exists a sequence of rational algebraic operations that transforms the original addition theorem polynomial equation into another polynomial equation whose root is a given branch of $\phi(u+v)$ and whose coefficients $P_r(u,v)$ all satisfy the equation 
$$P_r(u+k,v-k)\equiv P_r(u,v)$$ for arbitrary small $k$.  

\end{quote}
\end{lemma}

\begin{proof} Suppose that the addition theorem is 
\begin{equation}
\label{at0}
G\{\phi(u),\phi(v),\phi(u+v)\}=0
\end{equation}If we expand it into powers of $\phi(u+v)$, the equation \eqref{at0} takes the form (say):
\begin{equation}
\label{at01}
\phi^{m_1}(u+v)+P_{11}\left[\phi(u),\phi(v)\right]\phi^{m_1-1}(u+v)+\cdots =0
\end{equation}where the $P$'s are rational functions of $\phi(u)$ and $\phi(v)$.

Now comes the fundamental insight.  In \eqref{at01} we write
$$u+k_1 \ \text{for} \ u, \ \ \ \ v-k_1 \  \text{for} \ v,$$where $k_1$ is a small number.  Then, by this substitution $u+v$ \emph{remains unchanged} and the equation \eqref{at01} becomes 
\begin{equation}
\label{at02}
\phi^{m_1}(u+v)+P_{11}\left[\phi(u+k_1),\phi(v-k_1)\right]\phi^{m_1-1}(u+v)+\cdots =0.
\end{equation}Both polynomial equations \eqref{at01} and \eqref{at02} have the \emph{common root} $\phi(u+v).$  Therefore the two polynomials have a \emph{greatest common divisor} (GCD) whose degree $m_2\leq m_1$.

If the GCD has degree $m_1$ we stop the process.

If the GCD has degree $m_2<m_1$, then it has the form
\begin{equation}
\label{at11}
\phi^{m_2}(u+v)+P_{21}\left[\phi(u),\phi(u+k_1),\phi(v),\phi(v-k_1)\right]\phi^{m_1-1}(u+v)+\cdots =0
\end{equation}where the $P$'s are rational functions of their arguments.  In \eqref{at11} we write
$$u+k_2 \ \text{for} \ u, \ \ \ \ v-k_2 \  \text{for} \ v,$$where $k_2$ is a small number.  Then, by this substitution $u+v$ \emph{remains unchanged} and the equation \eqref{at11} becomes
\begin{equation}
\label{at12}
\phi^{m_2}(u+v)+P_{21}\left[\begin{aligned}
\phi(u+k_2),\phi(u+k_1+k_2)\\
\phi(v-k_2),\phi(v-k_1-k_2)
\end{aligned}
\right]\phi^{m_2-1}(u+v)+\cdots =0
\end{equation}Both polynomial equations \eqref{at11} and \eqref{at12} have the \emph{common root} $\phi(u+v).$  Therefore the two polynomials have a \emph{greatest common divisor} (GCD) whose degree $m_3\leq m_2$.

If the GCD has degree $m_2$ we stop the process.

If the GCD has degree $m_3<m_2$, then we repeat the GCD process.  

After a finite number of steps we either have $m_k=1$, or the two equations through which further reduction is made possible have \emph{all} of their roots in common.  Thus, in either case, the degree of the GCD cannot be lowered any further by this process.  

The form of the typical coefficient in this GCD is 

\begin{equation}
\label{at12}
P_{r}\left[\begin{aligned}
\phi(u),\phi(u+k_1),\phi(u+k_2),\cdots\phi(u+k_r)\cdots\phi(u+k_1+\cdots+k_r)\\
\phi(v),\phi(v-k_1),\phi(v-k_2),\cdots\phi(v-k_r)\cdots\phi(v-k_1-\cdots-k_r)
\end{aligned}
\right]
\end{equation}It follows that all the coefficients of the final GCD remain unaltered when $u$ is increased by a certain quantity $k$ and $v$ is decreased by the same quantity.

\end{proof}

We have proven that 
$$P_r(u+k,v-k)\equiv P_r(u,v)$$ for arbitrary small $k$.

This means that $P_r(u,v)$ satisfies the partial differential equation 
$$\frac{\partial P_r}{\partial u}=\frac{\partial P_r}{\partial v}$$and this latter has the general solution
$$P_r(u,v)\equiv \psi_r(u+v)$$for some function $\psi_r$.  Since $P_r(u,v)$ is a rational combination of the values of $\phi$, it is algebroid in any bounded domain.  Now \textsc{Schwarz} proves the following theorem:

\begin{theorem}
\begin{quote} \

\begin{enumerate}

  \item The function $\psi_r(u)$ is \textbf{meromorphic} in all the finite plane.
  \item The function $\psi_r(u)$ admits an \textbf{algebraic addition theorem} .
  \item The function $\phi(u)$ is the root of an algebraic equation whose coeficients are rational functions of the meromorphic function $\psi_r(u).$
\end{enumerate}

\end{quote}
\end{theorem}

\begin{proof} [Proof of (i)]

 Since $\psi_r(u)$ is algebroid in any bounded domain, it is sufficient to prove that it is \emph{single-valued} there.  In particular, it suffices to prove that $\psi_r(u)$ is single-valued in any circle of arbitrary radius centered at the origin $u=0$.
 
 By assumption, $\phi(u)$ is regular in the neighborhood of the origin.  Let $R$ be an arbitrary number and let  $K_R$ be a circle of radius R centered at the origin.
 
 We will remove from $K_R$ two narrow strips that are perpendicular to each other and such that no branch point of $\phi(u)$ is within the strips.  \emph{Here is where we use the fact that the singular points of $\phi(u)$ form a finite discrete set in $K_R$.} 
 
 There are infinitely many pairs of mutually perpendincular lines through the origin which do not intersect any singular point of $\phi(u)$.  We use these two diamaters of $K_R$ as mid-lines of strips which also do not have any singular points of $\phi(u)$ in their interiors, which again is possible since the set of singular points in $K_R$ is finite.
 
 The two strips form a cross-shaped figure and no branch point of $\phi(u)$ is within this cross.  The functions $\phi(u)$ and $\phi(v)$ are regular (and, in particular, single-valued)  along the mid-lines of the strips that form the cross.  
 
 \emph{We now take the} $k$'s \emph{defined above so small} that $$|k_1|+|k_2|+\cdots+|k_r|<\text{half of the width of the more narrow of the two strips.}$$
 
 Then if $u$ moves along the middle line of one of the strips, while $v$ moves along the middle line of the other, all of the artuments which have been used in the formation of $P_r$ are situated within the cross.  Thus, if $u$ and $v$ are added geometrically, we conclude that $P_r(u,v):=\psi_r(u+v)$ is a \emph{single-valued function} for all values of $u+v$ within the square which circumscribes the circle $K_R$.  Since R is arbitrarily large, this completes the proof of $1.$
 \end{proof}\

\begin{proof}[Proof of (ii)]
 
 If we write $v=0$ in $\psi_r(u+v)$, we obtain the equation
 \begin{equation}
\label{coef1}
\psi_r(u)=P_r\left[\begin{aligned}
\phi(u),\phi(u+k_1),\phi(u+k_2),\cdots\phi(u+k_r)\cdots\phi(u+k_1+\cdots+k_r)\\
\phi(0),\phi(0-k_1),\phi(0-k_2),\cdots\phi(0-k_r)\cdots\phi(0-k_1-\cdots-k_r)
\end{aligned}
\right]
\end{equation}Since $P_r$ is a rational function of its arguments, and since $\phi(u)$ admits an algebraic addition theorem, we see that $\phi(u+k_1)$ is an algebraic function of $\phi(u)$ and $\phi(k_1)$, and similarly for the other arguments.

Therefore there exists an algebraic equation of the form
 \begin{equation}
\label{af}
H\left[\phi(u),\psi_r(u)\right]=0.
\end{equation}

Therefore there are four algebraic equations:
\begin{align*}
\label{}
  &G\{\phi(u),\phi(v),\phi(u+v)\}=0   \\
    & H\left[\phi(u),\psi_r(u)\right]=0\\
    &H\left[\phi(v),\psi_r(v)\right]=0 \\
    &H\left[\phi(u+v),\psi_r(u+v)\right]=0
    \end{align*}
    
 We eliminate    $\phi(u),\phi(v),\phi(u+v)$ from them and the resultant is \begin{equation}
\label{ }
g\{\psi_r(u).\psi_r(v),\psi_r(u+v)\}=0
\end{equation}which shows that $\psi_r(u)$ \emph{admits an algebraic addition theorem}.
\end{proof}\

\begin{proof}[Proof of (iii)]

The equation \eqref{af} shows that $\phi(u)$ is an algebraic function of $\psi_r(u)$ and this is the last part of the theorem.

\end{proof}



\section{Functions of Several Variables}

Weierstrass~\cite{Weierstrass1} defined the concept of an \textbf{\emph{Algebraic Addition Theorem}} for analytic functions of $n$ variables as follows: Let\begin{equation}
\label{ }
x_1:=\varphi_1(\mathbf{u}), \cdots,x_n:= \varphi_1(\mathbf{u})
\end{equation}where $\mathbf{u}:=(u_1,\cdots,u_n)$, $\mathbf{u}\in \C^n$ in complex $n$-space $\C^n$ in complex 
Osgood space $\Bar\C^n$ (the cartesian product of $n$ Riemann spheres) be $n$ \emph{analytically independent }analytic function elements.  Suppose that in the neighborhood of $\mathbf{a}:=(a_1,\cdots,a_n)$ all of these function elements are holomorphic and that for $k=1,\cdots,n$\begin{equation}
\label{xk}
x_k=b_k+b_{k1}(u_1-a_1)+\cdots+b_{kn}(u_n-a_n)+\cdots
\end{equation}where we assume that the determinant of the linear terms, $|b_{kl}|\neq 0$, $k,l=1,\cdots,n.$  Then, by the inverse function theorem, we can revert the series~\eqref{xk} to obtain the $n$ series\begin{equation}
\label{uk}
u_k=a_k+c_{k1}(x_1-b_1)+\cdots+c_{kn}(x_n-b_n)+\cdots
\end{equation}which converge if $|x_1-b_1|,\cdots,|x_n-b_n|$ are sufficiently small.

Now set \begin{align*}
\label{}
  x_k:=& \varphi_k(\mathbf{u})   \\
y_k:=& \varphi_k(\mathbf{v})   \\ 
 z_k:=&\varphi_k(\mathbf{u+v}). 
\end{align*}where $\mathbf{u}.\mathbf{v},\mathbf{u+v}$ all lie in a sufficiently small neighborhood of $\mathbf{a}$ such that also the series\begin{align}
\label{ukvk}
u_k=&a_k+c_{k1}(x_1-b_1)+\cdots+c_{kn}(x_n-b_n)+\cdots\\
v_k=&a_k+c_{k1}(y_1-b_1)+\cdots+c_{kn}(y_n-b_n)+\cdots
\end{align}converge.  

If we substitute these expansions for $u_k$ and $v_k$ into the functions $\varphi_k(\mathbf{u+v})$ we obtain $n$ equations of the form \begin{equation}
\label{aat}
z_k=F_k(x_1,\cdots,x_n;y_1,\cdots,y_n)\equiv F_k(\mathbf{x};\mathbf{y})
\end{equation}where the coefficients of the functions $F_k$ are independent of the arguments $\mathbf{u}$ and $\mathbf{v}$.
\\
\begin{definition}
If these $n$ equations~\eqref{aat} all reduce to \emph{algebraic equations}
\begin{equation}
\label{ }
G_k(z_k;\mathbf{x};\mathbf{y})\equiv G_k[ \varphi_k(\mathbf{u+v}); \varphi_1(\mathbf{u}),\cdots, \varphi_n(\mathbf{u});\varphi_1(\mathbf{v}),\cdots, \varphi_n(\mathbf{v})]=0
\end{equation}we say that the system of $n$ functions $\varphi_k(\mathbf{u})$ \emph{\textbf{admits an algebraic addition theorem.}}
\end{definition}

Weierstrass announced the following theorem:
\begin{theorem}\
\begin{quote}
A set of $n$ analytically independent analytic functions $\varphi_1(\mathbf{u}), \cdots,\varphi_1(\mathbf{u})$ admits an algebraic addition theorem if, and only if, each is an algebraic function of the same  $n$ analytically independent \textbf{meromorphic} functions which admit an algebraic addition theorem.
\end{quote}
\end{theorem}

This year (2017), \textsc{Juan de Vicente} (\cite{Juan}) presented the first published proof of this theorem.  He gives a direct proof of the rational form of the addition theorem based on the \textsc{Hurwitz-Weierstrass} characterization of rational functions of several complex variables and the \textsc{Schwarz} idea (see above) of making the multi-valued mapping which admits an AAT dependent on coefficients in the polynomial which express the AAT.

 Weierstrass~\cite{Weierstrass1} proposed the problem of determining 
all meromorphic mappings \linebreak
$\Phi\: \C^n \to \Bar\C^n$ of complex $n$-space $\C^n$ into complex 
Osgood space $\Bar\C^n$ (the cartesian product of $n$ Riemann spheres) 
that admit an \textit{algebraic addition theorem} (\AAT). He announced 
the solution to be \textit{the set of all proper or singular abelian 
mappings}, i.e., those mappings
$$
\mathbf{\Phi(\mathbf{u})} := \bigl( \varphi_1(\mathbf{u}),\dots,\varphi_n(\mathbf{u}) \bigr), 
  \qquad  \mathbf{u}\in \C^n,
$$
where $ \varphi_1(\mathbf{u}),\dots,\varphi_n(\mathbf{u})$ are proper or singular abelian 
functions, i.e., meromorphic functions with $k\leq 2n$ periods.  Fifteen years later, Painlev\'e~\cite{Painleve} proved him 
right for the case $n = 2$, and in 1948 and 1954 
Severi~\cite{Severi}, and in 2008 Abe~\cite{Abe} proved the general case.

Unfortunately, the methods we employed for the one variable case apparently do not generalize to the case of abelian functions.

Painleve~\cite{Painleve} also claimed to have proved the theorem that any multi-valued analytic mapping with an AAT is algebraically dependent on $n$ independent meromorphic functions with an AAT, for the case $n=2$.  However Abe~\cite{Abe} disputes its validity.


\subsection*{Acknowledgments}
I thank Joseph~C. V\'arilly for 
many discussions and unrelenting encouragement. Support from the 
Vicerrector\'{\i}a de Investigaci\'on of the University of Costa Rica 
is acknowledged.


\end{document}